\documentclass[11pt]{article}
\usepackage{mathptmx}
\usepackage{times} 

\usepackage{amsmath,amssymb,color}
\usepackage{enumerate}

\usepackage{graphicx}
\usepackage{epstopdf}
\usepackage{subfigure}
\graphicspath{{figures/}}
\usepackage{graphicx}   
\usepackage{multirow}
\usepackage{amsfonts}
\usepackage{latexsym}
\usepackage{verbatim}
\usepackage{amsmath,authblk}
\usepackage{algorithm, algorithmic}

\usepackage{multirow}
\newtheorem{assumption}{Assumption}
\newtheorem{theorem}{Theorem}
\newtheorem{example}{Example}
\newtheorem{lemma}{Lemma}

\newtheorem{proposition}{Proposition}

\usepackage{epsf}
\usepackage{graphicx}
\usepackage{mathdots}
\usepackage{bbm}

\renewcommand{\mathbb}{\mathbbm}
\newcommand{\mathscr}{\mathcal}
\newcommand{\Rb}{\mathbb{R}}
\newcommand{\Eb}{\mathbb{E}}
\newcommand{\Fc}{\mathcal{F}}
\newcommand{\Nc}{\mathcal{N}}

\renewcommand{\epsilon}{\varepsilon}

\def\varkappa{\kappa}

\DeclareMathOperator*{\argmin}{arg\,min}

\newcommand{\beq}{\begin{equation}}
\newcommand{\eeq}{\end{equation}}
\newcommand{\beqa}{\begin{eqnarray}}
\newcommand{\eeqa}{\end{eqnarray}}
\newcommand{\beqas}{\begin{eqnarray*}}
\newcommand{\eeqas}{\end{eqnarray*}}
\newcommand{\bi}{\begin{tightitemize}}
\newcommand{\ei}{\end{tightitemize}}
\newcommand{\ba}{\begin{array}}
\newcommand{\ea}{\end{array}}

\newcommand{\nn}{\nonumber}

\def\vgap{\vspace*{.1in}}

\newenvironment{tightitemize}{%
    \list{{\textup{$\bullet$}}}{\settowidth\labelwidth{{\textup{\qquad}}}
    \leftmargin\labelwidth \advance\leftmargin\labelsep
    \parsep 0pt plus 1pt minus 1pt \topsep 3pt \itemsep 3pt
    }}{\endlist}

\usepackage{amsfonts}

\usepackage{epstopdf}
\usepackage{algorithmic}
\ifpdf
  \DeclareGraphicsExtensions{.eps,.pdf,.png,.jpg}
\else
  \DeclareGraphicsExtensions{.eps}
\fi

\title{A Single Time-Scale Stochastic Approximation Method for Nested Stochastic Optimization}

\author{Saeed Ghadimi\thanks{Department of Operations Research and Financial Engineering, Princeton University, Princeton, NJ 08544; email: sghadimi@princeton.edu}}
\author {Andrzej Ruszczy\'nski\thanks{Department of Management Science and Information Systems, Rutgers University, Piscataway, 08854; email: rusz@rutgers.edu}}
\author {Mengdi Wang\thanks{Department of Operations Research and Financial Engineering, Princeton University, Princeton, NJ 08544; email:mengdiw@princeton.edu}}
\affil{}

\begin{document}
\maketitle

\begin{abstract}
{
We study constrained nested stochastic optimization problems in which the objective function is a composition of two smooth functions whose exact values and
derivatives are not available. We propose a single time-scale stochastic approximation algorithm, which we call the Nested Averaged Stochastic Approximation (NASA), to find an approximate stationary point of the problem. The algorithm has two auxiliary averaged sequences (filters) which estimate the gradient of the composite objective function and the inner function value. By using a special Lyapunov function, we show that NASA achieves the sample complexity of ${\cal O}(1/\epsilon^{2})$ for finding an $\epsilon$-approximate stationary point, thus outperforming all extant methods for nested stochastic approximation. Our method and its analysis are the same for both unconstrained and constrained problems, without any need of batch samples for constrained nonconvex stochastic optimization. We also present a simplified parameter-free variant of the NASA method for solving constrained single level stochastic optimization problems, and we prove the same complexity result for both unconstrained and constrained problems.
}
\end{abstract}

\section{Introduction}

The main objective of this work is to propose a new recursive stochastic algorithm for constrained smooth composition optimization problems of the following form:
\begin{equation}
\label{main_prob}
\min_{x\in X} \big\{ F(x) = f(g(x)) \big\}.
\end{equation}
Here, the functions $f:\Rb^m\to\Rb$ and $g:\Rb^n\to\Rb^m$ are continuously differentiable,  and the set $X \subseteq \Rb^n$ is convex and closed. We do not assume $f,g$, or $F$ to be convex.

We focus on the simulation setting where neither the values nor the derivatives of $f$ or $g$ can be observed, but at any argument values $x\in \Rb^n$ and $u\in \Rb^m$ we can obtain random estimates of $ g(x)$,
of the Jacobian  $\nabla g(x)$, and of the gradient $ \nabla f(u)$. Such situations occur in \emph{stochastic composition optimization}, where we need to solve the problem:
\begin{equation}
\label{stoch_prob}
\min_{x\in X}\;  \mathbb{E}\big[\varphi\big( \mathbb{E}[\psi(x;\zeta)]; \xi\big)\big],
\end{equation}
in which $\zeta$ and $\xi$ are random vectors, and $\mathbb{E}$ denotes the expected value. In such situations, one can obtain samples
$(\tilde{\xi},\tilde{\zeta})$ of $(\xi,\zeta)$, and treat $\psi(x,\tilde{\zeta})$, $\nabla_x\psi(x,\tilde{\zeta})$, and $\nabla_u \varphi(u,\xi)$
as random estimates of $\mathbb{E}[\psi(x;\zeta)]$, $\nabla \mathbb{E}[\psi(x;\zeta)]$, and $\nabla\mathbb{E}\big[ \varphi(u,\xi)\big]$, respectively.
In this paper, we propose stochastic gradient-type methods for finding approximate stationary points of problems of this type. We also
derive sample complexity guarantees for these methods.

Stochastic composition problems of form \eqref{main_prob}--\eqref{stoch_prob} occur in many applications; we present three modern motivating examples.

\begin{example}[\emph{Stochastic Variational Inequality}]
\label{e:SVI}
{\rm
We have a random mapping $H:\Rb^n\times \Omega\to \Rb^n$
on some probability space $(\varOmega,\Fc,P)$ and a closed convex set $X$.  The problem is to find ${x}\in X$ such that
\begin{equation}
\label{SVI}
\big\langle \mathbb{E}[H(x)], \xi - {x} \big\rangle  \le 0, \quad \text{for all}\quad \xi \in X.
\end{equation}
The reader is referred to the recent publications \cite{iusem2017extragradient} and \cite{koshal2013regularized}
for a discussion of the challenges associated with this problem and its applications {\color{black} (our use of the ``$\le$'' relation instead of the common ``$\ge$''
is only motivated by the easiness to show the conversion to our formulation)}.
We propose to convert problem \eqref{SVI} to the nested form \eqref{main_prob} by defining the {\color{black}\emph{lifted gap function}} $f: \Rb^n \times \Rb^n\to \Rb$ as
\beq
f(x,h) = \max_{\xi \in X}\, \left\{ \langle h, \xi - x \rangle - \frac{1}{2} \| \xi - x\|^2 \right\},\label{gad_function}
\eeq
and the function $g:\Rb^n \to \Rb^n \times \Rb^n$ as $g(x) = \big(x,\mathbb{E}[H(x)]\big)$.
In this case, we actually have access to the gradient of $f$, but the value and the Jacobian of $g$ must be estimated. We do not require
${E}[H(x)]$ to be monotone.
\hfill $\Box$
}
\end{example}

\begin{example}[\emph{Policy Evaluation for Markov Decision Processes}]
\label{e:e2}
{\rm For a Markov chain $\{X_0,X_1,\ldots\}\subset \mathcal{X} $  with an unknown transition operator $P$, a reward function
$r:\mathcal{X} \mapsto \mathbb{R}$, and a discount factor  $\gamma\in(0,1)$, we want to estimate the value function $V:\mathcal{X} \mapsto \mathbb{R}$ given by $V(x) = \mathbb{E}\left[  \sum^{\infty}_{t=0} \gamma^t r(X_t) \mid X_0=x \right]$. For a finite space $\mathcal{X}$,
 {\color{black} the functions $r$ and $V$ may be viewed as vectors, and the following policy evaluation equation is satisfied:
\[
V = r + \gamma P V.
\]
As $P$ is not known and $|\mathcal{X}|$ may be large, this system cannot be solved directly.
To reduce the dimension of this problem, we employ a sketching matrix $S\in \Rb^{d\times |\mathcal{X}|}$ and a linear model for the value function $V(x) \approx \sum^k_{i=1} w_i \phi_i(x)$, where $\phi_1(\cdot),\ldots, \phi_k(\cdot)$ are given basis functions. The we can formulate the residual minimization problem for the policy evaluation
equation:
$$\min_{w\in\mathbb{R}^d}  \big\| S \left(\varPhi w - r - \gamma\, \mathbb{E}[{\hat P} ]\varPhi w\right) \big\|^2,
$$
where $\varPhi$ is the matrix with columns being the basis functions, and $\hat P$ is a sample transition matrix (see
\cite{WaLiFa17} and the references therein). In this case,
we may define the outer function $f$ as the squared norm,  and the inner function $g$ as the linear mapping inside the norm. Neither of the functions has an easily available value or derivative, but their samples can be generated by simulation.}\hfill $\Box$
}
\end{example}

\begin{example}[\emph{Low-Rank Matrix Estimation}]
\label{e:e3}
{\rm Let $\bar X\in \mathbb{R}^{n\times n} $ be an unknown matrix that we aim to approximate. One can sample from the unknown matrix and each sample returns a random matrix $X$ such that $\mathbb{E}[X] =\bar X$. Let $k<n$ be a pre-specified rank. The low-rank matrix estimation problem has the following form:
\[
\min_{(U,V)\in S} \ell\left( \mathbb{E}[X] - UV^T \right).
\]
In this problem, the unknowns $U$ and $V$ are $n\times k$ matrices, $S \subset \mathbb{R}^{n\times k} \times \mathbb{R}^{n\times k}$ is a bounded set, and $\ell:\mathbb{R}^{n\times n} \to \mathbb{R}$ is a loss function (\emph{e.g.}, the Frobenius norm).
Low-rank matrix approximation finds wide applications including image analysis, topic models, recommendation systems, and Markov models
(see \cite{ge2015escaping} and the references therein). Our formulation is nonconvex. When data arrive sequentially, our method can be applied in an on-line fashion to find a stationary solution.
\hfill $\Box$
}
\end{example}

Interest in stochastic approximation algorithms for problems of form \eqref{main_prob} dates back to \cite[Ch. V.4]{ermoliev1976methods}, where
penalty functions for stochastic constraints and composite regression models were considered.
There, and in the literature that followed, the main approach was to use two- or multiple-level stochastic
recursive algorithms in different time scales. For problems of form \eqref{main_prob} this amounts to using two stepsize sequences: one for updating the
main decision variable $x$, and another one for filtering the value of the inner function $g$. The crucial requirement is that
the outer method must be infinitely slower than the inner method, which decreases the convergence rate and creates practical difficulties.
Sample complexity analysis and acceleration techniques of stochastic approximation methods with multiple time scales  for solving problems of form \eqref{main_prob} gained interests in recent years.
We refer the readers to \cite{wang2017stochastic,WaLiFa17,yang2018multi} for a detailed account of these techniques and existing results for the general nested composition optimization problem. Furthermore, a Central Limit Theorem for the stochastic composition problem \eqref{main_prob}-\eqref{stoch_prob} has been established in  \cite{dentcheva2017statistical}. It shows that the $n$-sample empirical optimal value of problem \eqref{stoch_prob} converges to the true optimal value at a rate of ${\cal O}(1/\sqrt{n})$. {\color{black} The work \cite{ermoliev2013sample} establishes large deviation bounds for the empirical optimal value.}

In addition to the general solution methods studied in \cite{dentcheva2017statistical,wang2017stochastic,WaLiFa17,yang2018multi}, several notable special cases of the composition problem have been considered in the machine learning literature.
In the case where $f$ is convex and $g$ is linear, one can solve \eqref{main_prob} using duality-based methods via a Fenchel dual reformulation \cite{dai2017learning}. In the case where it is allowed to take mini-batches, Ref. \cite{blanchet2017unbiased} proposed a sampling scheme to obtain unbiased sample gradients of $F$ by
forming randomly sized mini-batches.
In the case when $f$ and $g$ take the form of a finite sum and strong convexity holds, one can leverage the special structure to obtain linearly convergent algorithms; see e.g. \cite{lian2017finite,liu2018dualityfree}.
To the authors' best knowledge, no method exists for solving \eqref{main_prob} with general smooth $f$ and $g$, which uses a single time-scale stochastic approximation update and does not resort to mini-batches. There is also no method for approximating stationary solutions that has provable complexity bounds, when the composition problem is constrained.

Our contributions are the following. First, we propose a new Nested Averaged Stochastic Approximation (NASA) algorithm for solving \eqref{main_prob}, which is qualitatively different from the earlier approaches. Its main idea, inspired by \cite{ruszczynski1983stochastic,ruszczynski87}, is to lift the problem into a higher dimensional space, $\Rb^n \times \Rb^n \times \Rb^m$, where our objective is not only to find the optimal $x$, but also to find the
 gradient of $F$ at the optimal point, and the value of $g$ at this point. In this space, we construct an iterative method using one stepsize sequence, and we prove convergence by employing a specially tailored merit (Lyapunov) function. This leads to the first single-timescale stochastic approximation algorithm for the composition problem, and entails essential improvements over the earlier approaches.

Secondly, we show that with proper choice of the stepsize sequence, ${\cal O}(1/\epsilon^2)$ observations are sufficient for NASA algorithm to find
a pair $(\bar x, \bar z) \in X \times \Rb^n$ satisfying $\Eb[V(\bar x, \bar z)] \le \epsilon$, where $\bar z$ is an estimate for $\nabla F(\bar x)$, and
$V(x,z)$ is an optimality measure generalizing $\|\nabla F(\bar x)\|^2$ to constrained problems; see \eqref{Vx}. This complexity bound is consistent with the Central Limit Theorem for composite risk functionals~\cite{dentcheva2017statistical} and
is better than the best-known complexity of ${\cal O}(1/\epsilon^{2.25})$ obtained in \cite{WaLiFa17} for smooth nested nonconvex stochastic optimization. In fact, our complexity bound for the two-level composition problem is of the same order as the complexity of the stochastic gradient method for general smooth one-level nonconvex stochastic optimization~\cite{GhaLan12}.

Thirdly, our convergence analysis of the NASA method is the same for both unconstrained and constrained problems, allowing us to obtain the same complexity of ${\cal O}(1/\epsilon^2)$ for the constrained case, without taking batches of samples per iterations. To the best of our knowledge, this is the first direct convergence analysis of {\color{black} a method for general stochastic nested problems of form \eqref{main_prob} which avoids} multiple samples per iteration to reduce the variance of the stochastic gradients. Hence, this property makes the NASA method attractive for online learning where the samples are received one by one.

Finally, we present a simplified variant of the NASA method for solving a class of single-level stochastic optimization problem, \emph{i.e.}, with $g(x) \equiv x$ in problem \eqref{main_prob}. This simplified version is a form of a constrained dual averaging method. We show that the sample (iteration) complexity of this algorithm is of the same order as that of the NASA method. Moreover, the stepsize schedule of this method, unlike almost all existing stochastic approximation algorithms, does not depend on any problem parameters or employ line-search procedures. Its rate of convergence is established without forming mini-batches of samples, and is valid for both unconstrained and constrained problems. It should be mentioned that a similar complexity bound has recently been obtained in \cite{DavBen18,davis2019stochastic} for finding an approximate stationary point (albeit with a different optimality measure) for nonsmooth, constrained, and nonconvex one-level stochastic optimization problems without taking mini-batches of samples per iteration.
{\color{black} Some online algorithms for constrained problems that use mini-batches of samples choose their size to improve error complexity, while making a trade-off with sample complexity (see, \emph{e.g.}, \cite{GhaLan16,GhaLanZhang16}). Hence, they may be more desirable in the cases where the projection onto the feasible set is computationally hard.}

{\bf Notation.} The optimal value of problem \eqref{main_prob} is denoted by $F^*$. For any Lipschitz continuous function $h$, we use $L_h$ to denote its Lipschitz constant. We use $\nabla h$ to denote the gradient (or Jacobian) of a scalar (or vector) function $h$.

\section{The Method}
Our goal in this section is to propose a stochastic approximation algorithm for solving problem \eqref{main_prob} where estimates  of the gradient
of $f$ and the value and Jacobian of $g$ are available through calls to a  stochastic oracle.

The method generates three random sequences, namely, approximate solutions $\{x^k\}$,
average gradients $\{z^k\}$, and average $g$-values $\{u^k\}$, defined
on a certain probability space $(\Omega,\Fc,P)$. We let $\Fc_k$ to be
the $\sigma$-algebra generated by
$\{x^0,\dots,x^k,z^0,\dots,z^k, u^0\dots,u^k\}$.
We also make the following assumption on the stochastic oracle.

\begin{assumption}\label{stoch_assump}
For each $k$, the stochastic oracle delivers random vectors $G^{k+1}\in \Rb^m$, $s^{k+1} \in \Rb^n$, and a random matrix $J^{k+1} \in \Rb^{m \times n}$, such that $s^{k+1}$ and $J^{k+1}$ are conditionally independent given $\Fc_k$, and
\begin{align*}
\Eb[G^{k+1}|\Fc_k] &= g(x^{k+1}), \qquad  \Eb[\|G^{k+1}-g(x^{k+1})\|^2 | \Fc_k] \le \sigma^2_G, \\
\Eb[J^{k+1}|\Fc_k] &= \nabla g(x^{k+1}), \qquad \Eb[\|J^{k+1}\|^2 | \Fc_k] \le \sigma^2_J,\\
\Eb[s^{k+1}|\Fc_k] &= \nabla f(u^k), \qquad \Eb[\|s^{k+1}\|^2 | \Fc_k] \le \sigma^2_s.
\end{align*}
\end{assumption}
The shift in indexing of $x^k$ in the above assumption is due to the fact that $x^{k+1}$ will be $\Fc_k$-measurable in our method. Our method  proceeds as presented in Algorithm \ref{alg_NASA}.

\begin{algorithm} [h]
	\caption{Nested Averaged Stochastic Approximation (NASA)}
	\label{alg_NASA}
	\begin{algorithmic}

\STATE \emph{Input:}
$x^0 \in X$, $z^0 \in \Rb^n$, $u^0 \in \Rb^m$,  $a>0$, $b>0$. 

\STATE 0. Set $k=0$.

\STATE 1. {\color{black} For $\Fc_k$-measurable regularization coefficient $\beta_k >0$ and stepsize $\tau_k \in (0,1/a]$ ,} compute
\begin{equation}
\label{QP}
y^k = \argmin_{y \in X}\  \left\{\langle z^k, y-x^k \rangle + \frac{\beta_k}{2} \|y-x^k\|^2\right\},
\end{equation}
and set
\beq \label{def_xk}
x^{k+1} = x^k + \tau_k (y^k-x^k).
\eeq

\STATE 2. Call the stochastic oracle to obtain $s^{k+1}$ at $u^k$, $G^{k+1}$ and $J^{k+1}$ at $x^{k+1}$, and update the running averages as
\begin{align}
z^{k+1} &= (1-a\tau_k)z^k + a\tau_k \big[J^{k+1}\big]^T s^{k+1},\label{def_zk}\\
u^{k+1} &= (1-b\tau_k)u^k + b\tau_k G^{k+1}.\label{def_uk}
\end{align}

\STATE 3. Increase $k$ by one and go to Step 1.

	\end{algorithmic}
\end{algorithm}

\vgap

{A few remarks are in order. First, the stochastic gradient $ \big[J^{k+1}\big]^T s^{k+1}$ returned by the stochastic oracle is a biased estimator of the gradient of $F(x^{k+1})$. Hence, $z^{k+1}$, as a weighted average of these stochastic gradients, is also a biased estimator of $\nabla F(x^{k+1})$. However, the sum of the bias terms of the latter estimator grows slower than the former one, ensuring convergence of the algorithm (see Theorem~\ref{cvrg_rate_nocvx}). Secondly, $u^{k+1}$ is also a biased estimator of $g(x^{k+1})$, whose error can be properly controlled and asymptotically driven to $0$. Finally, convergence of Algorithm~\ref{alg_NASA} depends on the choice of the sequences $\{\tau_k\}$ and $\{\beta_k\}$ and the parameters $a$ and $b$, which will be specified in the next section.
}

We end this section with a brief review of the optimality conditions for problem \eqref{main_prob} and their relation to the subproblem \eqref{QP}.
The following result is standard (see, e.g., \cite[Thm. 3.24]{ruszczynski2006nonlinear}.
\begin{theorem}
\label{opt-cond_theom}
If a point $\hat{x}\in X$ is a local minimum of problem \eqref{main_prob}, then
\begin{equation}
\label{opt-cond}
-\nabla F(\hat{x}) \in \Nc_X(\hat{x}),
\end{equation}
where $\Nc_X(\hat{x})$ denotes the normal cone to $X$ at the point $\hat{x}$. If in addition the function $F(\cdot)$ is convex, then every point $\hat{x}$ satisfying \eqref{opt-cond}
is the global minimum of problem \eqref{main_prob}.
\end{theorem}

Condition \eqref{opt-cond} is closely related to the subproblem \eqref{QP}. Denote (for $\beta>0$)
\[
 \bar{y}(x,z,\beta) = \argmin_{y \in X}\  \left\{\langle z, y-x \rangle + \frac{\beta}{2} \|y-x\|^2\right\}.
\]
Elementary manipulation shows that
\[
\bar{y}(x,z,\beta) = \Pi_X\Big(x - \frac{1}{\beta}z\Big),
\]
where $\Pi_X(\cdot)$ is the operation of the orthogonal projection on the set $X$. The relation $-z\in \Nc_X(x)$ is equivalent to $\bar{y}(x,z,\beta) = x$.
We will, therefore, use the function
\begin{equation}
\label{Vx}
V(x,z) = \|\bar{y}(x,z,1) - x\|^2 + \| z - \nabla F(x)\|^2
\end{equation}
as a measure of violation of the optimality condition \eqref{opt-cond} by the primal-dual pair $(x,z)$.
In the unconstrained case where $X=\mathbb{R}^n$,  we have $V(x,z) =\|z\|^2 + \|z- \nabla F(x)\|^2$.

The following lemma relates the optimality measure $V(x,z)$ to subproblem \eqref{QP} for an arbitrary $\beta>0$.

\begin{lemma}
\label{l:opt-bound}
For every $x\in X$ and every $\beta>0$,
\begin{equation}
\label{opt-bound}
\|\bar{y}(x,z,1) - x\| \le \max(1,\beta) \;\|\bar{y}(x,z,\beta) - x\|.
\end{equation}
\end{lemma}
\begin{proof}
To simplify notation, set $x=0$,  $y(\beta) = \bar{y}\big(x,z,\beta\big) = \Pi_X\Big(-\frac{1}{\beta}z\Big)$.
 By the
characterization of the orthogonal projection,
\begin{equation}
\label{projection}
\big \langle  z + \tau y(\tau), \xi - y(\tau) \big \rangle \ge 0 \quad \text{for all} \quad \xi \in X,\quad \tau>0.
\end{equation}
Setting $\tau=\beta$ and $\xi = y(1)$ we obtain
\[
\big \langle  z + \beta y(\beta), y(1) - y(\beta) \big \rangle \ge 0.
\]
Now we set $\tau=1$ and $\xi = y(\beta)$ in \eqref{projection} and get
\[
\big \langle  z +  y(1),  y(\beta) - y(1) \big \rangle \ge 0.
\]
Adding these inequalities yields
\begin{equation}
\label{sum-ineq}
\langle \beta y(\beta) - y(1), y(1)-y(\beta) \rangle \ge 0.
\end{equation}
Consider two cases.

\noindent
\emph{Case 1:} $\beta \ge 1$.
Inequality \eqref{sum-ineq} implies that
\[
(\beta-1)\|y(\beta)\| = \| \beta y(\beta)  - y(\beta)\| \ge \|y(1) - y(\beta)\|.
\]
By the triangle inequality and the last relation,
\[
\|y(1)\| \le \| y(\beta)\| + \|y(1)- y(\beta)\| \le \beta \|y(\beta)\|.
\]
which proves our claim in this case.

\noindent
\emph{Case 2:} $0 < \beta \le 1$.
From \eqref{sum-ineq} we obtain
\[
(1-\beta) \langle y(\beta) - y(1), y(1) \rangle \ge \beta\|y(\beta) - y(1) \|^2 \ge 0.
\]
Therefore, $\| y(\beta)\| \ge \|y(1)\|$ and our claim is true in this case as well.
\end{proof}

{\color{black}
Note that our measure of non-optimality in \eqref{Vx} is an upper  bound for the squared norm of the gradient, when $X=\mathbb{R}^n$. For the constrained case, it can also be related to the existing ones in the literature. To do so, we need to view other algorithms in the primal-dual space. For example, for the proximal point mapping used in \cite{DavBen18,davis2019stochastic},
\[
\hat{y} = \argmin_{y\in X} \left\{ F(y) + \frac{1}{2} \|y - x\|^2\right\},
\]
the squared distance $\|\hat{y}-x\|^2$ is employed as a measure of non-optimality (the use of a parameter $\gamma$ there can be dealt with by minor adjustments). By optimality conditions of the above subproblem, we have
$\hat{y} = \bar{y}(x,\hat{z},1)$ with $\hat{z}=\nabla F(\hat{y})$. If we view the method of \cite{DavBen18} as generating
primal--dual pairs of form $(x,\hat{z})$, we obtain
\[
\| \hat{y}- x\|^2 \le V(x,\hat{z}) = \| \hat{y}- x\|^2 + \| \hat{z} -\nabla F(x)\|^2 \le (1+ L_{\nabla F}^2) \| \hat{y}- x\|^2.
\]
It follows that both optimality criteria would be equivalent in the primal--dual space, if $\hat{z}$ were observed. In the (accelerated) projected gradient method of \cite{GhaLan16,GhaLanZhang16}, the optimality criterion is the squared distance $\|\tilde{y}-x\|^2$, where
\[
\tilde{y} = \argmin_{y\in X} \left\{ \langle \nabla F(x), y \rangle + \frac{1}{2} \|y - x\|^2\right\}.
\]
Evidently,  $\tilde{y} = \bar{y}(x,\tilde{z},1)$ and if we could see the dual vector $\tilde{z}=\nabla F(x)$ we would obtain
\[
\| \tilde{y}- x\|^2 = V(x,\tilde{z}) = \| \hat{y}- x\|^2 + \| \tilde{z} -\nabla F(x)\|^2.
\]
It should be mentioned that while the above $\hat z$ and $\tilde z$ are not computable under stochastic setting,
the vector $z_k$ defined in \eqref{def_zk} is computed every iteration and can be used as an online estimate of $\nabla F(x)$.
}

\section{Convergence Analysis}
In this section, we provide convergence analysis of Algorithm~\ref{alg_NASA}. To do so, we need the following assumption.

\begin{assumption} \label{f-g-assum}
The functions $f$ and $g$ and their derivatives are Lipschitz continuous.
\end{assumption}

\vgap

This immediately implies that the gradient of the composite function $F$ is Lipschitz continuous.

\begin{lemma} \label{f_lips}
Under Assumption~\ref{f-g-assum}, the gradient of the function $F$ defined in \eqref{main_prob} is Lipschitz continuous with $L_{\nabla F}:= L_g^2 L_{\nabla f}+L_f L_{\nabla g}$.
\end{lemma}

\begin{proof}
Let $x,\hat x \in X$ be given. Then, by the chain rule we have
\begin{align*}
\|\nabla F(x)-\nabla F(\hat x)\| &= \|\nabla g(x)^\top \nabla f(g(x)) - \nabla g(\hat x)^\top \nabla f(g(\hat x))\| \nn \\
&\le \|\nabla g(x)\| \|\nabla f(g(x))-\nabla f(g(\hat x))\| + \|\nabla f(g(\hat x))\| \|\nabla g(x)-\nabla g(\hat x)\|\\
& \le (L_g^2 L_{\nabla f}+L_f L_{\nabla g}) \|x -\hat x\|.
\end{align*}
\end{proof}

\vgap
The next result about the subproblem employed at Step 2 of Algorithm~\ref{alg_NASA} will be used in our convergence analysis.

\begin{lemma} \label{eta_lips}
Let $\eta(x,z)$ be the optimal value of subproblem \eqref{QP} for any $(x,z)$, \emph{i.e.},
\beq \label{QP_g}
\eta(x,z) = \min_{y \in X}\  \left\{\langle z, y-x \rangle + \frac{\beta}{2} \|y-x\|^2\right\}.
\eeq
Then the gradient of $\eta$ w.r.t. $(x,z)$ is Lipschitz continuous with the constant
\[L_{\nabla \eta} = 2\sqrt{(1+\beta)^2+(1+\tfrac{1}{2\beta})^2}.\]
\end{lemma}

\begin{proof}
Let $\bar y(x,z) \in X$ be the solution of \eqref{QP_g}. Since the solution is unique, the partial derivatives of the optimal value function $\eta$ are given by
\[
\nabla_x \eta(x,z) = -z+\beta(x-\bar y(x,z)), \ \ \nabla_z \eta(x,z) = \bar y(x,z)-x.
\]
Hence, for any  $(x,z)$ and $(\hat x,\hat z)$, we have
\beqa
\|\nabla \eta(x,z) - \nabla \eta(\hat x,\hat z)\| &\le& \|\nabla_x \eta(x,z) - \nabla_x \eta(\hat x,\hat z)\| + \|\nabla_z \eta(x,z) - \nabla_z \eta(\hat x,\hat z)\| \nn \\
&\le& 2(1+\beta)\|x-\hat x\|+ (2+1/\beta)\|z-\hat z\| \le L_{\nabla \eta}\|(x,z)-(\hat x,\hat z)\|,\nn
\eeqa
where the inequalities follow from the nonexpansiveness of the projection operator and the Cauchy-Schwarz inequality, respectively.
\end{proof}

The proof of convergence of Algorithm~\ref{alg_NASA} follows from the analysis of the following merit function:
\beq \label{def_merit}
W(x,z,u) = a (F(x)-F^*) - \eta(x,z) + \frac{\gamma}{2}\| g(x)-u\|^2  ,
\eeq
where $\gamma>0$ and $\eta(x,z)$ is the optimal value of subproblem \eqref{QP_g}.

\begin{lemma} \label{main_convergence}
Let $\{x^k,z^k,y^k,u^k\}_{k \ge 0}$ be the sequence generated by Algorithm~\ref{alg_NASA}. Also assume that Assumption~\ref{f-g-assum} holds, and
\beq \label{const_cond}
\beta_k = \beta>0 \ \ \forall k \ge 0, \qquad 2(a \beta-c)(\gamma b-2c) \ge L_g^2(a L_{\nabla f}+\gamma)^2
\eeq
for some positive constants $c$ and $\gamma$. Then
\beq \label{main_rec}
c \sum_{k=0}^{N-1} \tau_k \left(\|d^k\|^2 +\|g(x^k)-u^k\|^2 \right) \le W(x^0,z^0,u^0)+\sum_{k=0}^{N-1} r^{k+1} \qquad \forall N \ge 1,
\eeq
where, for any $k \ge 0$,
\begin{align}
 d^k &= y^k-x^k,\nn\\
r^{k+1}&= \frac{\tau_k^2}{2} \Big([a L_{\nabla F}+ L_{\nabla \eta}+ \gamma L_g^2+2 a L_g^2 L_{\nabla f}] \|d^k\|^2 + b^2\|g(x^{k+1})-G^{k+1}\|^2 \Big)\nn \\
&{\quad} +\tau_k\Big(\gamma b(1-b\tau_k) \langle g(x^{k+1})-u^k, \Delta^g_k\rangle  + a \langle d^k,\Delta^F_k\rangle \Big)+\frac{L_{\nabla \eta}}{2} \|z^{k+1}-z^k\|^2,\nn \\
\Delta^g_k =& g(x^{k+1})-G^{k+1}, \qquad \Delta^F_k :=  \nabla g(x^{k+1})^\top \nabla f(u^k)-\big[J^{k+1}\big]^\top s^{k+1}. \label{def_rk}
\end{align}

\end{lemma}

\begin{proof}
We estimate the decrease of the three terms of the function $W(x,z,u)$ in iteration $k$. \\
\textbf{1.} Due to Assumption~\ref{f-g-assum} and in view of Lemma~\ref{f_lips}, we have
\[
F(x^k) - F(x^{k+1}) \ge \langle \nabla F(x^{k+1}), x^k - x^{k+1} \rangle -\frac{L_{\nabla F}}{2} \|x^k - x^{k+1}\|^2.
\]
After re-arranging the terms and using \eqref{def_xk}, we obtain
\beq \label{f_bnd}
F(x^{k+1})- F(x^k) \le \tau_k \langle \nabla F(x^{k+1}), d^k \rangle +\frac{ L_{\nabla F}\tau_k^2}{2} \|d^k\|^2.
\eeq
\textbf{2.} By \eqref{def_xk}, \eqref{def_zk}, and  Lemma~\ref{eta_lips}, we have
\beqa\label{eta_bnd1}
\eta(x^k,z^k)- \eta(x^{k+1},z^{k+1}) &\le& \langle z^k+ \beta_k(y^k-x^k), x^{k+1}-x^k \rangle- \langle y^k-x^k, z^{k+1}-z^k\rangle \nn \\
&+&\frac{L_{\nabla \eta}}{2} \left[\|x^{k+1}-x^k\|^2 +\|z^{k+1}-z^k\|^2 \right]\nn \\
&=& \tau_k \langle (1+a)z^k+ \beta_k d^k, d^k \rangle- a \tau_k \langle d^k, \big[J^{k+1}\big]^T s^{k+1}\rangle \nn \\
&+&\frac{L_{\nabla \eta}}{2} \left[\|x^{k+1}-x^k\|^2 +\|z^{k+1}-z^k\|^2 \right].
\eeqa
Due to the optimality condition of subproblem \eqref{QP}, we have $\langle z^k+\beta_k(y^k-x^k), y-y^k \rangle \ge 0 $, for all  $y \in X$, which together with the choice of $y=x^k$ implies that
\beq \label{opt_QP}
\langle z^k, d^k \rangle +\beta_k \|d^k\|^2 \le 0.
\eeq
Combining the last relation with \eqref{eta_bnd1}, we obtain
\beqa\label{eta_bnd2}
\eta(x^k,z^k)- \eta(x^{k+1},z^{k+1}) &\le& -a\beta_k \tau_k \|d^k\|^2- a \tau_k \langle d^k, \nabla g(x^{k+1})^\top \nabla f(u^k)\rangle \nn \\
&+& a \tau_k \langle d^k, \nabla g(x^{k+1})^\top \nabla f(u^k)-\big[J^{k+1}\big]^\top s^{k+1}\rangle\nn\\
&+& \frac{L_{\nabla \eta}}{2} \left[\|x^{k+1}-x^k\|^2 +\|z^{k+1}-z^k\|^2 \right].
\eeqa
\textbf{3.} By \eqref{def_uk} we have
\beqa
\|g(x^{k+1})-u^{k+1}\|^2 &=& (b \tau_k)^2\|g(x^{k+1})-G^{k+1}\|^2+ (1-b\tau_k)^2\|g(x^{k+1})-u^k\|^2\nn \\
&+&2b\tau_k(1-b\tau_k) \langle g(x^{k+1})-G^{k+1}, g(x^{k+1})-u^k\rangle, \nn \\
\|g(x^{k+1})-u^k\|^2 &\le&\|g(x^k)-u^k\|^2+ L_g^2 \tau_k^2 \|d^k\|^2+2L_g\tau_k\|d^k\|\|g(x^k)-u^k\|,\label{g_bnd}
\eeqa
where the last inequality follows from \eqref{def_xk} and the Lipschitz continuity of $g$. Moreover, using \eqref{def_xk} and Assumption~\ref{f-g-assum}, we obtain
\beq \label{error_F}
\langle d^k, \nabla F(x^{k+1})-\nabla g(x^{k+1})^\top \nabla f(u^k)\rangle \le L_g L_{\nabla f}\left[\tau_k L_g\|d^k\|^2+\|d^k\|\|g(x^k)-u^k\|\right].
\eeq
\textbf{4.} The overall estimate is obtained by combining \eqref{def_merit}, \eqref{f_bnd}, \eqref{eta_bnd2}, \eqref{g_bnd}, and \eqref{error_F}:
\begin{align}
&W(x^{k+1},z^{k+1},u^{k+1})-W(x^{k},z^{k},u^{k}) \nn \\
&\le -\tau_k \Big( a \beta_k \|d^k\|^2 +\frac{\gamma b}{2} \|g(x^k)-u^k\|^2  - (a L_g L_{\nabla f}+\gamma L_g) \|d^k\| \|g(x^k)-u^k\| \Big) + r^{k+1},\nn
\end{align}
where $r^{k+1}$ is defined in \eqref{def_rk}. Hence, when condition \eqref{const_cond} holds, we have
\[
W(x^{k+1},z^{k+1},u^{k+1})-W(x^{k},z^{k},u^{k}) \le - c \tau_k \left( \|d^k\|^2 +\|g(x^k)-u^k\|^2 \right) + r^{k+1}.
\]
Observe that $\eta(x,z) \le 0$ for any $(x,z)$, due to \eqref{opt_QP}. Therefore, $W(x,z,u) \ge 0$ for all $(x,z,u)$. Summing up both sides of the above inequalities and re-arranging the terms, we obtain \eqref{main_rec}.
\end{proof}

\vgap

As a consequence of the above result, we can provide upper bounds for the sequences generated by Algorithm~\ref{alg_NASA}.

\begin{proposition}\label{prop_conv}
Let $\{x^k,z^k,y^k,u^k\}_{k \ge 0}$ be the sequence generated by Algorithm~\ref{alg_NASA} and Assumption~\ref{stoch_assump} hold. Then:\\
{\rm \textbf{(a)}} If $\tau_0 =1/a$, we have
\[
\beta_k^2 \Eb\big[\|d^k\|^2 | \Fc_{k-1}\big] \le \Eb\big[\|z^k\|^2 | \Fc_{k-1}\big] \le \sigma^2_J \sigma^2_s \qquad \forall k \ge 1;
\]
{\rm \textbf{(b)}} If Assumption~\ref{f-g-assum} also holds, {\color{black}$\beta_k =\beta>0$ for all $k$ and $a \tau_k \le 1/\sqrt{2}$ for all $k \ge 1$}, we have
\beqa
\sum_{k=0}^\infty \Eb[\|z^{k+1}-z^k\|^2 | \Fc_k] &\le& 2{\color{black}\left[\|z^0\|^2+24 a^2 \sigma^2_J \sigma^2_s \sum_{k=0}^\infty \tau_k^2 \right]},\nn \\
\sum_{k=0}^\infty \Eb[r^{k+1} | \Fc_k] &\le& \sigma^2 \sum_{k=0}^\infty \tau_k^2,\label{sum_rk}
\eeqa
where
\beq
\sigma^2= \frac{1}{2} \Big([L_{\nabla F}+ L_{\nabla \eta}+ \gamma L_g^2+2 a L_g^2 L_{\nabla f}]\frac{\sigma^2_J \sigma^2_s}{\beta^2} + b^2 \sigma_g^2 + 4 L_{\nabla \eta} {\color{black}\left[\|z^0\|^2+24 a^2 \sigma^2_J \sigma^2_s\right]}\Big).\label{rk_bnd}
\eeq

\end{proposition}

\begin{proof}
We first show part (a). The first inequality follows immediately from \eqref{opt_QP} and the Cauchy-–Schwarz inequality. Also, defining
\beq \label{def_Gamma}
\Gamma_1 := \left\{
\begin{array}{ll}
 1, & \tau_0 = 1/a,\\
1 - a\tau_0, & \tau_0<1/a,
\end{array} \right. \ \
\Gamma_{k} := \Gamma_1 \prod_{i=1}^{k-1} (1 - a\tau_i)  \ \ \forall k \ge 2,
\eeq
and noting \eqref{def_zk}, we obtain
\[
\frac{z^1}{\Gamma_1} = \frac{(1 - a\tau_0)z^0}{\Gamma_1}+ \frac{a \tau_0}{\Gamma_1}  \big[J^1\big]^\top s^1, \qquad  \frac{z^{k+1}}{\Gamma_{k+1}} = \frac{z^k}{\Gamma_k}+ \frac{a \tau_k}{\Gamma_{k+1}}  \big[J^{k+1}\big]^\top s^{k+1} \quad \forall k \ge 1.
\]
Summing up the above inequalities and assuming that $\tau_0 =1/a$, we obtain, for any $k \ge 1$,
\beq \label{def_zk2}
z^k = \sum_{i=0}^{k-1} \alpha_{i,k} \big[J^{i+1}\big]^\top s^{i+1}, \quad \alpha_{i,k} = \frac{a \tau_i}{\Gamma_{i+1}}\Gamma_k, \quad \sum_{i=0}^{k-1} \alpha_{i,k} = 1,
\eeq
where the last equality follows from the fact that
\[
\sum_{i=0}^{k-1} \frac{a \tau_i}{\Gamma_{i+1}} = \frac{a \tau_0}{\Gamma_1} + \sum_{i=1}^{k-1} \frac{a \tau_i}{(1-a\tau_i)\Gamma_i} = \frac{a \tau_0}{\Gamma_1} + \sum_{i=1}^{k-1} \left(\frac{1}{\Gamma_{i+1}}-\frac{1}{\Gamma_{i}}\right) = \frac{1}{\Gamma_{k}}-\frac{1-a \tau_0}{\Gamma_1}.
\]
Therefore, noting that $\|\cdot\|^2$ is a convex function and using Assumption~\ref{stoch_assump}, we conclude that
\[
\Eb\big[\|z^k\|^2 \big| \Fc_{k-1}\big] \le \sum_{i=0}^{k-1} \alpha_{i,k} \Eb\big[\|J^{i+1}\|^2 \big| \Fc_i\big]\Eb\big[\|s^{i+1}\|^2 \big| \Fc_i\big] \le \sigma^2_J \sigma^2_s \sum_{i=0}^{k-1} \alpha_{i,k} =\sigma^2_J \sigma^2_s.
\]
We now show part (b). By \eqref{def_zk}, {\color{black} the above estimate, and assuming that $\tau_0 =1/a$}, we have
\beqa
{\color{black}\Eb\big[\|z^1-z^0\|^2 \big| \Fc_0\big]} &\le& {\color{black}2 \Big(\|z^0\|^2+\Eb\big[\|[J^1]^\top s^1\|^2 \big| \Fc_0\big] \Big)\le 2 \Big(\|z^0\|^2+\sigma^2_J \sigma^2_s\Big)},\nonumber \\
\Eb\big[\|z^{k+1}-z^k\|^2 \big| \Fc_k\big] &\le& \frac{2 {\color{black}a^2} \tau_k^2}{{\color{black}(1-a\tau_k)^2}} \Big({\color{black} \Eb\big[\|z^{k+1}\|^2 \big| \Fc_{k}\big]}+\Eb\big[\|[J^{k+1}]^\top s^{k+1}\|^2 \big| \Fc_k\big] \Big)\nonumber \\
&\le& \frac{4 {\color{black}a^2} \sigma^2_J \sigma^2_s \tau_k^2}{{\color{black}(1-a\tau_k)^2}} \qquad \forall k \ge 1,\nonumber
\eeqa
implying that
\[
{\color{black}\sum_{k=0}^\infty \Eb[\|z^{k+1}-z^k\|^2 | \Fc_k] \le 2\left[\|z^0\|^2+\sigma^2_J \sigma^2_s \left(1+2a^2 \sum_{k=1}^\infty  \big(\frac{\tau_k}{1-a\tau_k}\big)^2\right)\right]}.
\]
{\color{black}Combining the above inequality with the fact that $\tfrac{1}{(1-a \tau_k)^2} \le 12$ due to the assumption that $a \tau_k \le 1/\sqrt{2}$,} we obtain the first inequality in (b). Finally, due to the equation
\[
\Eb\big[\langle g(x^{k+1})-u^k, \Delta^g_k\rangle \big| \Fc_k\big] = \Eb\big[\langle d^k,\Delta^F_k\rangle  \big| \Fc_k\big] =0,
\]
the second inequality in  (b) follows from the first one in (a) and \eqref{def_rk}.
\end{proof}

\vgap

We are now ready to estimate the quality of the iterates generated by Algorithm~\ref{alg_NASA}. In view of Lemma~\ref{l:opt-bound},
we can bound the optimality measure at iteration $k$ as follows:
\beq\label{def_Vopt}
V(x^k,z^k) \le  \max(1,\beta_k^2) \; \| d^k\|^2 + \|z^k - \nabla F(x^k)\|^2.
\eeq

\begin{theorem} \label{cvrg_rate_nocvx}
Suppose Assumptions  \ref{stoch_assump} and \ref{f-g-assum} are satisfied and  let $\{x^k,z^k,y^k,u^k\}_{k \ge 0}$ be the sequence generated by Algorithm~\ref{alg_NASA}. Moreover, assume that {\color{black}the parameters are chosen so that \eqref{const_cond} holds and stepsizes $\{\tau_k\}$ {\color{black} are deterministic} and satisfy}
\beq\label{const_cond1}
\sum_{i=k+1}^{N} \tau_i \Gamma_i \le \bar c \Gamma_{k+1} \qquad \forall k \ge 0 \ \ \text{and} \ \ \forall N \ge 1,  \ \ {\color{black} \text{and} \ \ a \tau_k \le 1/\sqrt{2} \ \ \forall k \ge 1,}
\eeq
where $\Gamma_k$ is defined in \eqref{def_Gamma}, and $\bar c$ is a positive constant. Then:\\
{\rm \textbf{(a)}} For every $N \ge 2$, we have
\begin{equation}
\label{error_grad}
\begin{aligned}
\sum_{k=1}^N \tau_k \Eb\big[\|\nabla F(x^k)-z^k\|^2 \big|\Fc_{k-1}\big] &\le a \bar c \left(\frac{1}{c}\max(L_1,L_2)\sigma^2+2 a \sigma^2_J \sigma^2_s  \right) \bigg(\sum_{k=0}^{N-1} \tau^2_k\bigg)  \\
&{\quad} + \frac{a \bar{c}}{c}\max(L_1,L_2)W(x^0,z^0,u^0),
\end{aligned}
\end{equation}
where
\beq\label{def_L12}
L_1:= \frac{2L_{\nabla F}^2}{a^2} + 4 L_g^4 L_{\nabla f}^2, \qquad L_2:= 4L_g^2 L_{\nabla f}^2.
\eeq
{\rm \textbf{(b)}} As a consequence, we have
\begin{multline}
\Eb \big[V(x^R,z^R)\big]
\le \frac{1}{\sum_{k=1}^{N-1} \tau_k}\bigg\{
a \bar c \left(\frac{1}{c}\big[\max(L_1,L_2)+\max(1,\beta^2)\big]\sigma^2 +2 a \sigma^2_J \sigma^2_s  \right) \bigg(\sum_{k=0}^{N-1} \tau^2_k\bigg) \\
 +     \frac{1}{c}\Big(a \bar c \max(L_1,L_2)+\max(1,\beta^2)\Big)W(x^0,z^0,u^0) \bigg\},
 \label{main_rec2}
\end{multline}
{\color{black}where the expectation is taken with respect to all random sequences generated by the method and an independent random integer number $R \in \{1,\ldots,N-1\}$}, whose probability distribution is given by
\beq \label{def_probl}
P[R=k] = \frac{\tau_k}{\color{black} \sum_{j=1}^{N-1} \tau_j}.
\eeq
{\rm \textbf{(c)}} Moreover, if $a=b=1$, if the regularization coefficients are constant and equal to:
\beq \label{const_cond2}
\beta_k \equiv \beta = \left(\frac{(1+\alpha)^2}{\alpha} L_g^2+\frac{\alpha}{4}\right)L_{\nabla f},
\eeq
for some $\alpha>0$, and if the stepsizes are equal to:
\beq \label{def_tau}
\tau_0 = 1, \qquad \tau_k \equiv \frac{1}{\sqrt{N}} \qquad \forall k=1,\ldots,N-1,
\eeq
then
\begin{gather}
\label{main_rec3}
\Eb \big[V(x^R,z^R)\big]  \le \frac{4}{\sqrt{N}-1} \bigg(\frac{2}{\alpha L_{\nabla F}}\big[\max(L_1,L_2)
 +\max(1,\beta^2)\big]\big[W(x^0,z^0,u^0)+\sigma^2\big]+ \sigma^2_J \sigma^2_s \bigg),
\intertext{and}
\Eb \big[\|g(x^R)-u^R\|^2\big]  \le \frac{W(x^0,z^0,u^0)+\sigma^2}{\alpha L_{\nabla F}(\sqrt{N}-1)}.\label{main_rec4}
\end{gather}
\end{theorem}

\begin{proof}
We first show part (a). By \eqref{def_zk}, we have
\begin{multline*}
\nabla F(x^{k+1})-z^{k+1} \\
= (1-a\tau_k)[\nabla F(x^k)-z^k+\nabla F(x^{k+1})-\nabla F(x^k)]+a\tau_k[\nabla F(x^{k+1})-[J^{k+1}]^\top s^{k+1}].
\end{multline*}
Dividing both sides of the above inequality by $\Gamma_{k+1}$, summing them up, noting the fact that $\tau_0=1/a$, similar to \eqref{def_zk2}, we obtain
\begin{align} \label{Fk-zk}
&\nabla F(x^k)-z^k = \sum_{i=0}^{k-1} \alpha_{i,k} \left[e_i+ \Delta^F_i \right] \qquad \forall k \ge 1, \nn\\
\intertext{with}
& e_i := \frac{(1-a\tau_i)}{a\tau_i}\big[\nabla F(x^{i+1})-\nabla F(x^i)\big] + \nabla F(x^{i+1})-\nabla g(x^{i+1})^\top \nabla f(u^i),
\end{align}
where $\Delta^F_i$ is defined in \eqref{def_rk}. Hence,
\[
\nabla F(x^{k-1})-z^{k-1} = \sum_{i=0}^{k-2} \alpha_{i,k-1} \left[e_i+ \Delta^F_i \right]=\frac{\Gamma_{k-1}}{\Gamma_k}\sum_{i=0}^{k-2} \alpha_{i,k} \left[e_i+ \Delta^F_i \right],
\]
which together with \eqref{Fk-zk} implies that
\begin{align*}
\|\nabla F(x^k)-z^k\|^2 &= \Big\|\frac{\Gamma_k}{\Gamma_{k-1}}\big[\nabla F(x^{k-1})-z^{k-1}\big]+\alpha_{k-1,k} \left[e_{k-1}+ \Delta^F_{k-1} \right]\Big\|^2 \nn\\
&= \big\|(1-a\tau_{k-1})\big[\nabla F(x^{k-1})-z^{k-1}\big]+a\tau_{k-1} \left[e_{k-1}+ \Delta^F_{k-1} \right] \big\|^2 \nn \\
&=\big\|(1-a\tau_{k-1})\big[\nabla F(x^{k-1})-z^{k-1}\big]+a\tau_{k-1}e_{k-1}\big\|^2+ a^2\tau^2_{k-1} \left\|\Delta^F_{k-1}\right\|^2 \nn \\
&{\quad} + 2a\tau_{k-1}\big\langle (1-a\tau_{k-1})\big[\nabla F(x^{k-1})-z^{k-1}\big]+a\tau_{k-1}e_{k-1}, \Delta^F_{k-1}\big\rangle \nn \\
&\le (1-a\tau_{k-1}) \big\|\nabla F(x^{k-1})-z^{k-1}\big\|^2+a\tau_{k-1} \left\|e_{k-1}\|^2+ a^2\tau^2_{k-1}\|\Delta^F_{k-1}\right\|^2\nn \\
&{\quad} + 2a\tau_{k-1}\big\langle (1-a\tau_{k-1})[\nabla F(x^{k-1})-z^{k-1}]+a\tau_{k-1}e_{k-1}, \Delta^F_{k-1}\big\rangle,
\end{align*}
where the inequality follows from the convexity of $\|\cdot\|^2$. Dividing both sides of the above inequality by $\Gamma_k$, using \eqref{def_Gamma}, summing all the resulting inequalities, and noting the facts that $\tau_0=1/a$, and $\tau_k \le 1/a$, we obtain
\beq \label{Fz-zk2}
\|\nabla F(x^k)-z^k\|^2 \le \Gamma_k \left[\sum_{i=0}^{k-1} \left(\frac{a\tau_i}{\Gamma_{i+1}}\|e_i\|^2+ \frac{a^2\tau^2_i}{\Gamma_{i+1}}\|\Delta^F_i\|^2+
\frac{2a\tau_i \delta_i}{\Gamma_{i+1}}\right)
\right],
\eeq
where
\[
\delta_i:= \langle (1-a\tau_i)[\nabla F(x^i)-z^i]+a\tau_i e_i, \Delta^F_i\rangle.
\]
Now, using \eqref{def_xk}, \eqref{error_F}, with a view to Lemma~\ref{f_lips}, we obtain
\[
\|e_i\|^2 \le \frac{2L_{\nabla F}^2(1-a\tau_i)^2}{a^2}\|d^i\|^2 +  4 L_g^2 L_{\nabla f}^2\Big[\tau_i^2 L_g^2\|d^i\|^2+\|g(x^i)-u^i\|^2\Big],
\]
which together with \eqref{def_L12} implies that
\begin{equation}
\label{bound_ei}
\begin{aligned}
\lefteqn{\sum_{k=1}^{N} \bigg(\tau_k \Gamma_k \sum_{i=0}^{k-1} \frac{a\tau_i}{\Gamma_{i+1}} \|e_i\|^2\bigg) \le
\sum_{k=1}^{N} \bigg(\tau_k \Gamma_k \sum_{i=0}^{k-1} \frac{a \tau_i}{\Gamma_{i+1}} \left[L_1 \|d^i\|^2+L_2\|g(x^i)-u^i\|^2\right]\bigg)}\hspace{5em} \\
& =a \sum_{k=0}^{N-1} \left\{\frac{\tau_k}{\Gamma_{k+1}}\bigg[\sum_{i=k+1}^{N}\tau_i \Gamma_i\bigg]\bigg[L_1 \|d^k\|^2+L_2\|g(x^k)-u^k\|^2\bigg]\right\}\\
&\le a \bar c \sum_{k=0}^{N-1} \left\{\tau_k\left[L_1 \|d^k\|^2+L_2\|g(x^k)-u^k\|^2\right]\right\},
\end{aligned}
\end{equation}
where the last inequality follows from the condition \eqref{const_cond1}. In a similar way,
\beq \label{delta_bnd}
\sum_{k=1}^{N} \bigg(\tau_k \Gamma_k \sum_{i=0}^{k-1} \frac{a^2\tau_i^2}{\Gamma_{i+1}} \|\Delta^F_i\|^2\bigg)
= a^2 \sum_{k=0}^{N-1} \frac{\tau^2_k}{\Gamma_{k+1}}\bigg(\sum_{i=k+1}^{N}\tau_i \Gamma_i\bigg)
\|\Delta^F_k\|^2
\le \bar c a^2 \sum_{k=0}^{N-1} \tau^2_k
\|\Delta^F_k\|^2.
\eeq
Moreover, under Assumption~\ref{stoch_assump} we have
\[
\Eb[\|\Delta^F_k\|^2 | \Fc_k] \le 2\Eb[\big[J^{k+1}\big]^\top s^{k+1}| \Fc_k] \le 2 \sigma^2_J \sigma^2_s, \qquad \Eb[\delta_k | \Fc_k]=0.
\]
Therefore, by taking the conditional expectation of both sides of \eqref{Fz-zk2}, noting \eqref{bound_ei}, \eqref{delta_bnd}, the above inequality, and Lemma~\ref{main_convergence}, we obtain \eqref{error_grad}. Part (b) then follows from \eqref{def_Vopt} and the facts that
\begin{align*}
&{\color{black}c \sum_{k=0}^{N-1} \tau_k \Eb[\|d^k\|^2 | \Fc_k] \le W(x^0,z^0,u^0)+\sigma^2 \sum_{k=0}^{N-1} \tau^2_k},\\
&\Eb \big[V(x^R,z^R)\big]={\frac {\sum_{k=1}^{N-1} \tau_k \Eb \big[V(x^k,z^k)\big]}{\sum_{k=1}^{N-1} \tau_k}},
\end{align*}
due to \eqref{main_rec}, \eqref{sum_rk}, and \eqref{def_probl}.

To show part (c), observe that condition \eqref{const_cond} is  satisfied by \eqref{const_cond2} and the choice of $\gamma=4c= \alpha L_{\nabla f}$. Also by \eqref{def_Gamma} and \eqref{def_tau}, we have
\begin{align*}
&{\sum_{k=1}^{N-1} \tau_k \ge \sqrt{N}-1, \qquad \sum_{k=0}^{N-1} \tau_k^2 \le 2}, \qquad \Gamma_k = \Big(1-\frac{1}{\sqrt{N}}\Big)^{k-1},\\
&\sum_{i=k+1}^{N} \tau_i \Gamma_i = \Big(1-\frac{1}{\sqrt{N}}\Big)^k \frac{1}{\sqrt{N}} \sum_{i=0}^{N-k-1} \Big(1-\frac{1}{\sqrt{N}}\Big)^i \le \Big(1-\frac{1}{\sqrt{N}}\Big)^k,
\end{align*}
which ensures \eqref{const_cond1} with $\bar c=1$ and hence, together with \eqref{main_rec2}, implies \eqref{main_rec3}.
\end{proof}

\vgap

We now add a few remarks about the above result.
First, the estimate \eqref{main_rec3} implies that to find an approximate stationary point $(\bar x,\bar z)$ of problem \eqref{main_prob} satisfying
$\Eb[V(\bar x,\bar z)] \leq\epsilon$, Algorithm~\ref{alg_NASA} requires at most ${\cal O}(1/\epsilon^2)$  iterations (stochastic gradients), which
is better than ${\cal O}(1/\epsilon^{2.25})$ obtained in \cite{WaLiFa17} for unconstrained nonconvex stochastic optimization. Our complexity bound indeed matches the sample complexity of the stochastic gradient method for the general (single-level) smooth nonconvex optimization problem \cite{GhaLan12}.
It is also consistent with the Central Limit Theorem for composite risk functionals~\cite{dentcheva2017statistical}, because the objective value gap is essentially proportional to the squared gradient norm at an approximate stationary point (in the unconstrained case).
Secondly, Theorem~\ref{cvrg_rate_nocvx} provides the same convergence rate of Algorithm~\ref{alg_NASA} for both constrained and unconstrained problems: we can still get sample complexity of ${\cal O}(1/\epsilon^2)$ with taking only one sample per iteration in the constrained case. To the best of our knowledge, this is the first direct convergence analysis for constrained nonconvex stochastic optimization providing the aforementioned sample complexity.
{Thirdly, \eqref{main_rec4} provides not only accuracy bounds for the approximate stationary point $x^R$ but also squared error bounds for estimating the exact value of the inner function, $g(x^R)$, by the second running average, $u^R$. As a result, Algorithm~\ref{alg_NASA} provides not only accurate approximations to the stationary solution but also reliable estimates of the gradients. {\color{black}Finally, note that Assumption~\ref{stoch_assump} implies that derivatives of $f$ and $g$ are bounded. Hence, to establish the results of Theorem~\ref{cvrg_rate_nocvx}, we can relax Assumption~\ref{f-g-assum} and require Assumption~\ref{stoch_assump} together with Lipschitz continuity of derivatives of $f$ and $g$.}

For a stochastic variational inequality problem of Example \ref{e:SVI}, our method has significantly lower complexity and faster
convergence than the approach of \cite{iusem2017extragradient}. \color{black}Most of literature on stochastic methods for SVI requires monotonicity or even strong monotonicity of the operator $H(\cdot)$ (see, \emph{e.g.},
\cite{juditsky2011solving,koshal2013regularized} and the references within). The authors in \cite{iusem2017extragradient} consider SVI with operators $\Eb[H(\cdot)]$ satisfying a weaker pseudo-monotonicity assumption. With the use of a variance reduction technique by increasing sizes of mini-batches, a generalization of the extragradient method originating in \cite{korpelevich1976extragradient} was developed, with oracle complexity
$\mathcal{O}(\varepsilon^{-2})$, which is matched by our results. The recent manuscript \cite{lin2018solving}, considers a special case of SVI, a stochastic saddle point problem, with weakly convex-concave functions. By employing a proximal point scheme, the resulting SVI is converted to a monotone one, and approximately solved by a stochastic algorithm. The resulting two-level scheme has rate of convergence $\mathcal{O}(\varepsilon^{-2})$, if acceleration techniques are used. Our approach does not make any monotonicity assumption and achieves oracle complexity of  $\mathcal{O}(\varepsilon^{-2})$ as well. This is due to the use of a one-level method for a special merit function \eqref{def_merit}, involving the lifted gap function \eqref{gad_function} as one of its components (the other being a projection mapping using its gradient).}

We also have the following asymptotic convergence result.

\begin{theorem}
\label{t:convergence}
Assume that the sequence of stepsizes satisfies
\beq\label{general_cond}
\sum_{k=1}^\infty \tau_k = + \infty \quad {\color{black} \text{a.s.}, \qquad \Eb} \sum_{k=1}^\infty \tau_k^2  < \infty.
\eeq
Then a constant $\bar{a}>0$ exists  such that, for all $a\in (0,\bar{a})$,  with probability 1, every accumulation point $(x^*,z^*,u^*)$ of the sequence $\{x^k,z^k,u^k)$ generated by Algorithm \ref{alg_NASA} satisfies the conditions:
\begin{gather*}
z^* = \big[\nabla g(x^*)\big]^T \nabla f(x^*),\\
u^* = g(x^*),\\
-z^* \in {\cal N}_X(x^*).
\end{gather*}
\end{theorem}
\begin{proof}
Note that the sequence $\{r^k\}$ defined in \eqref{def_rk} is adapted to $\{\Fc_k\}$ and summable almost surely under Assumption~\ref{stoch_assump} and \eqref{general_cond}. Therefore, \eqref{main_rec} implies that almost surely
\[
\lim_{k \to \infty} \inf \|d^k\|=0, \qquad \lim_{k \to \infty} \inf \|g(x^k)-u^k\|=0.
\]
Using the techniques of \cite[Lem. 8]{ruszczynski87}, we can then show that with probability~1, $d^k\to 0$ and $g(x^k)-u^k \to 0$. Following similar analysis to \cite{ruszczynski87}, we prove that each convergent subsequence of $\{(x^k,z^k,u^k)\}$ converges to a stationary point of problem \eqref{main_prob}, the corresponding gradient of $F$, and the value of $g$.
\end{proof}

\section{Dual averaging with constraints}

 Although our main interest is in composite optimization, we also
 provide a simplified variant of Algorithm~\ref{alg_NASA} for solving a single-level stochastic optimization. Our techniques allow the same convergence analysis for both constrained and uncostrained problems which removes the necessity of forming mini-batches of samples per iteration for constrained problems (see e.g., \cite{GhaLanZhang16, GhaLan16}). Moreover, this variant of the NASA method, different from the existing SA-type methods, is a parameter-free algorithm in the sense that its stepsize policy does not depend on any problem parameters {\color{black} and allows for random (history dependent) stepsizes}. This algorithmic feature is more important under the stochastic setting since estimating problem parameters becomes more difficult.

Throughout this section, we assume that the inner function $g$ in \eqref{main_prob} is the identity map, \emph{i.e.}, $g(x)\equiv x$, and only noisy information about $f$ is available. In this case, our problem of interest is reduced to
\begin{equation}
\label{main_prob_new}
\min_{x\in X} f(x),
\end{equation}
and it is easy to verify that
\beq \label{g_param}
G=g, \qquad J=I, \qquad \sigma_G=0, \qquad \sigma_J =1, \qquad L_g=1, \qquad L_{\nabla g}=0.
\eeq
Moreover, Algorithm~\ref{alg_NASA} can be simplified as follows.

\begin{algorithm} [H]
	\caption{The Averaged Stochastic Approximation (ASA) Method}
	\label{alg_ASA}
	\begin{algorithmic}

\STATE  Replace Step 2 of Algorithm~\ref{alg_NASA} with the following:\\

2'. Call the stochastic oracle to obtain $s^{k+1}$ at $x^k$ and update the ``running average'' as
\beq \label{def_zk_new}
z^{k+1} = (1-a\tau_k)z^k + a\tau_k s^{k+1}.
\eeq

	\end{algorithmic}
\end{algorithm}

\vgap

The above algorithm differs from Algorithm~\ref{alg_NASA} in two aspects. First, stochastic approximation of $\nabla g$ is replaced by its exact value, the identity matrix. Secondly, the averaged sequence in \eqref{def_uk} is not required and $u^{k}$ is simply set to $x^{k}$ due to the fact that exact function values of $g$ are available.

 The resulting method belongs to the class of algorithms with direction averaging (multi-step) methods. The literature on these methods
 for unconstrained problems is very rich. They were initiated in \cite{tsypkin1970fundamentals}, developed and analyzed in
 \cite{gupal1972stochastic,korostelev1981multi,polyak1977comparison} and other works. Recently, these methods play a role in
 machine learning, under the name of \emph{dual averaging methods}
 (see \cite{xiao2010dual} and the references therein).

 In all these versions, two time-scales were essential for
 convergence. The first single time-scale method was proposed in \cite{ruszczynski1983stochastic}, with convergence analysis based on a different
 Lyapunov function suitable for unconstrained problems. Our version is related to \cite{ruszczynski87}, where a similar approach was proposed,
 albeit without rate of convergence estimates.  {\color{black}We may remark here that the version of \cite{ruszczynski87} calculates
 $s^{k+1}$ at $x^{k+1}$ rather than at $x^k$ at Step 2', which is essential for nonsmooth weakly convex $f(\cdot)$.
 For smooth functions both ways are possible and can be analyzed in the same way with minor adjustments.}

Convergence analysis of Algorithm~\ref{alg_ASA} follows directly from that of Algorithm~\ref{alg_NASA}
by simplifying the definition of the merit function as
\beq \label{def_merit_new}
W(x,z) = a (f(x)-f^*) - \eta(x,z),
\eeq
exactly as used in \cite{ruszczynski87}.
We then have the following result.
\begin{lemma} \label{main_convergence_new}
Let $\{x^k,z^k,y^k\}_{k \ge 0}$ be the sequence generated by Algorithm~\ref{alg_ASA}. Also assume that function $f$ has Lipschitz continuous gradient and $\beta_k = \beta>0 \ \ \forall k \ge 0$. Then:\\
{\rm \textbf{(a)}} For any $N \ge 2$, we have
\beq \label{main_rec_new}
\beta \sum_{k=0}^{N-1} \tau_k \|d^k\|^2  \le W(x^0,z^0)+\sum_{k=0}^{N-1} r^{k+1},
\eeq
where, for any $k \ge 0$,
\begin{equation}
\label{def_rk_new}
\begin{aligned}
 &d^k = y^k-x^k, \qquad \Delta^f_k =  \nabla f(x^k)-s^{k+1}, \\
&r^{k+1}= \frac{1}{2}(3aL_{\nabla f}+ L_{\nabla \eta}) \tau_k^2\|d^k\|^2+a\tau_k\langle d^k,\Delta^f_k\rangle+\frac{1}{2}L_{\nabla \eta} \|z^{k+1}-z^k\|^2.
\end{aligned}
\end{equation}
{\rm\textbf{(b)}} If, in addition, Assumption~\ref{stoch_assump} holds along with \eqref{g_param}, with the choice of
{\color{black}$\beta_k =\beta>0$ for all $k \ge 0$, and $a \tau_k \le 1/\sqrt{2}$  for all $k \ge 1$}, and $\tau_0 =1/a$, we have
\begin{equation}
\begin{aligned}
&\beta^2 \Eb[\|d^k\|^2 | \Fc_{k-1}]] \le \Eb[\|z^k\|^2 | \Fc_{k-1}]] \le \sigma^2_s \qquad \forall k \ge 1,  \\
&\sum_{k=0}^\infty \Eb[\|z^{k+1}-z^k\|^2 | \Fc_k] \le 2{\color{black}\left[\|z^0\|^2+24 a^2 \sigma^2_s \sum_{k=0}^\infty \tau_k^2 \right]},  \\
&\sum_{k=0}^\infty \Eb[r^{k+1} | \Fc_k] \le \left(\frac{(3aL_{\nabla f}+ L_{\nabla \eta})\sigma^2_s }{2\beta^2}+ 2 L_{\nabla \eta}{\color{black}\left[\|z^0\|^2+24 a^2 \sigma^2_s \right]}\right) \sum_{k=0}^\infty \tau_k^2:=\sigma^2 \sum_{k=0}^\infty \tau_k^2.
\end{aligned}
\label{rk_bnd_new}
\end{equation}
\end{lemma}

\begin{proof}
{\color{black}Multiplying \eqref{f_bnd} by $a$, summing it up with \eqref{eta_bnd2}, noting \eqref{g_param}, \eqref{def_merit_new}, and the fact that
\[
a\tau_k\langle d^k,\nabla f(x^{k+1})-\nabla f(x^k)\rangle \le a L_{\nabla f} \tau_k^2 \|d^k\|^2,
\]
we obtain
\beq\label{eq_param_free}
W(x^{k+1},z^{k+1})-W(x^{k},z^{k}) \le -a \beta_k \tau_k  \|d^k\|^2+ r^{k+1}.
\eeq
The remainder of the proof is similar to that of Lemma~\ref{main_convergence_new} and Proposition~\ref{prop_conv}; hence, we skip the details.
}
\end{proof}

Using the above results, we can provide the main convergence property of Algorithm~\ref{alg_ASA}.

\begin{theorem} \label{t:main_convergence_new}
Let $\{x^k,z^k,y^k\}_{k \ge 0}$ be the sequence generated by Algorithm~\ref{alg_ASA}, $\beta_k = \beta>0 \ \ \forall k \ge 0$, and $a=1$. Moreover, assume that Assumption~\ref{stoch_assump} holds along with \eqref{g_param}, and the stepsizes are set to \eqref{def_tau}. Then we have
\beq \label{main_rec3_new}
\Eb \big[V(x^R,z^R)\big]  \le \frac{1}{\sqrt{N}-1} \left(\frac{1}{\beta}(\max(1,\beta^2)+L^2_{\nabla f})\big[W(x^0,z^0)+2\sigma^2\big]+ 4\sigma^2_s\right),
\eeq
where $V(x,z)$ is defined in \eqref{Vx}.
\end{theorem}

\begin{proof}
Similar to \eqref{Fk-zk}, we have
\beq \label{proof_new1}
\nabla F(x^k)-z^k = \sum_{i=0}^{k-1} \alpha_{i,k} \left[e_i+ \Delta^f_i \right], \quad
 e_i := \frac{\nabla f(x^{i+1})-\nabla f(x^i)}{a\tau_i},
\eeq
which together with the Lipschitz continuity of $\nabla f$ and \eqref{def_xk} imply that
\beq \label{proof_new2}
\|e_i\|^2 \le \frac{L^2_{\nabla f}\|d^i\|^2}{a^2}.
\eeq
In view of Lemma~\ref{main_convergence_new}, the rest of the proof is similar to that of Theorem~\ref{cvrg_rate_nocvx}.
\end{proof}

\vgap

It is worth noting that unlike Algorithm~\ref{alg_NASA}, the regularization coefficient $\beta_k$ in Algorithm~\ref{alg_ASA}, {\color{black}due to \eqref{eq_param_free}}, can be set to any positive constant number to achieve the sample (iteration) complexity of ${\cal O}(1/\epsilon^2)$.  Such a result has not been obtained before for a parameter-free algorithm for smooth nonconvex stochastic optimization. Moreover, Algorithm~\ref{alg_ASA}, similar to Algorithm~\ref{alg_NASA}, outputs a pair of $(x^R,z^R)$ where $z^R$ is an accurate estimate of
$\nabla f(x^R)$ without taking any additional samples. This is important for both unconstrained and constrained problems, where one can use the quantity
$\max(1,\beta_k)\|y^k - x^k\|$ as an online certificate of the quality of the current solution; see Lemma \ref{l:opt-bound}.

Note that the convergence result of Theorem~\ref{t:main_convergence_new} is established under the boundedness assumption of the second moment of the stochastic gradient. In the remainder of this section, we modify the convergence analysis of Algorithm~\ref{alg_ASA} under a relaxed assumption that only variance of the stochastic gradient is bounded. This assumption, which is common  in the literature on smooth stochastic optimization, is stated as follows.
\begin{assumption}\label{var_stoch_assump}
For each $k$, the stochastic oracle delivers a random vector $s^{k+1} \in \Rb^n$ such that
\[
\Eb[s^{k+1}|\Fc_k] = \nabla f(x^k), \qquad \Eb[\|s^{k+1}-\nabla f(x^k)\|^2 | \Fc_k] \le \hat \sigma^2_s.
\]
\end{assumption}

\begin{lemma} \label{main_convergence_new2}
Let $\{x^k,z^k,y^k\}_{k \ge 0}$ be the sequence generated by Algorithm~\ref{alg_ASA}. Also assume that the function $f$ has a Lipschitz continuous gradient
and $\beta_k = \beta>0$, for all $k \ge 0$. Then:\\
{\rm \textbf{(a)}} For any $N \ge 2$, we have
\begin{equation} \label{main_rec_new2}
\sum_{k=1}^{N-1} \tau_k \left(\beta_k-\tfrac{(3aL_{\nabla f}+ L_{\nabla \eta}) \tau_k}{2} \right)\|d^k\|^2  \le W(x^0,z^0)+\sum_{k=0}^{N-1} \hat r^{k+1},
\end{equation}
where, for any $k \ge 0$,
\begin{equation}\label{def_rk_new2}
\hat r^{k+1}= \big\langle a\tau_k d^k-a^2 \tau^2_k L_{\nabla \eta}(\nabla F(x^k)-z^k) ,\Delta^f_k\big\rangle
+ \frac{1}{2}a^2 \tau^2_k L_{\nabla \eta} \left[\|\nabla F(x^k)-z^k\|^2 + \|\Delta^f_k\|^2\right].
\end{equation}
{\rm\textbf{(b)}} If, in addition, Assumption~\ref{var_stoch_assump} holds and stepsizes $\{\tau_k\}$ are chosen such that
\beq\label{const_cond1_new}
\tau_0 =1/a, \qquad \sum_{i=k+1}^{N} \tau^2_i \Gamma_i \le \hat c \tau_k \Gamma_{k+1} \qquad \forall k \ge 0 \ \ \text{and} \ \ \forall N \ge 2,
\eeq
where $\Gamma_k$ is defined in \eqref{def_Gamma}, and $\hat c$ is a positive constant, we have
\begin{multline}
\label{dk_expect_new}
\sum_{k=0}^{N-1} \tau_k \left(\beta_k-\frac{1}{2}(3aL_{\nabla f}+ L_{\nabla \eta}+a \hat c L_{\nabla \eta} L^2_{\nabla f}) \tau_k \right)\Eb[\|d^k\|^2 | \Fc_k] \\
 \le W(x^0,z^0)+\frac{L_{\nabla \eta}}{2}\|\nabla F(x^0)-z^0\|^2+\frac{1}{2}a^2(L_{\nabla \eta}+2 \hat c)\hat \sigma^2_s \sum_{k=0}^{N-1} \tau_k^2.
\end{multline}
\end{lemma}

\begin{proof}
To show part (a), note that by \eqref{def_zk}, \eqref{g_param}, and \eqref{def_rk_new2}, we have
\[
\|z^{k+1}-z^k\|^2 = a^2 \tau^2_k \left[\|\nabla F(x^k)-z^k\|^2 + \|\Delta^f_k\|^2-2\langle \nabla F(x^k)-z^k, \Delta^f_k\rangle \right],
\]
which together with \eqref{main_rec_new} and in view of \eqref{def_rk_new2} implies \eqref{main_rec_new2}. To show part (b), note that by \eqref{proof_new1}, \eqref{proof_new2}, and similar to the proof of Theorem~\ref{cvrg_rate_nocvx}, part (a), we have
\begin{multline*}
\sum_{k=1}^{N} \tau^2_k \|\nabla F(x^k)-z^k\|^2 \\
\le \hat c \sum_{k=0}^{N-1} \tau^2_k \left(\frac{L_{\nabla f}^2\|d^k\|^2}{a} +a^2 \tau_k\|\Delta^f_k\|^2 + 2a\langle \nabla F(x^{k+1})-z^k-a\tau_i[\nabla F(x^k)-z^k], \Delta^f_k\rangle\right).
\end{multline*}
Taking conditional expectation from both sides of the above inequality and using \eqref{def_rk_new2} under Assumption~\ref{var_stoch_assump},  with the choice of $\tau_0=1/a$, we obtain
\begin{multline*}
\sum_{k=0}^{N-1} \Eb[r^{k+1} | \Fc_k] \\
\le \frac{a^2(L_{\nabla \eta}+2 \hat c)\hat \sigma^2_s}{2} \sum_{k=0}^{N-1} \tau_k^2+\frac{a \hat c L_{\nabla \eta} L^2_{\nabla f}}{2} \sum_{k=0}^{N-1} \tau^2_k \Eb[\|d^k\|^2 | \Fc_{k-1}]+\frac{L_{\nabla \eta}}{2}\|\nabla F(x^0)-z^0\|^2
\end{multline*}
(with the notation of $\Fc_{-1}\equiv \Fc_{0}$), which together with \eqref{main_rec_new2} implies \eqref{dk_expect_new}.
\end{proof}

We can now specialize the convergence rate of Algorithm~\ref{alg_ASA} by properly choosing the stepsize policies.

\begin{theorem} \label{t:main_convergence_new2}
Let $\{x^k,z^k,y^k\}_{k \ge 0}$ be the sequence generated by Algorithm~\ref{alg_ASA}, the gradient of function $f$ be Lipschitz continuous, and stepsizes set to \eqref{def_tau}. If Assumption~\ref{var_stoch_assump} holds and
\beq\label{def_beta_k_new}
\beta_k \equiv \beta \ge \frac{2(3L_{\nabla f}+ L_{\nabla \eta}+ \hat c L_{\nabla \eta} L^2_{\nabla f})}{3} \qquad k \ge 0,
\eeq
we have
\begin{multline}\label{main_rec3_new2}
\Eb \big[V(x^R,z^R)\big]  \le \frac{1}{\sqrt{N}-1} \bigg(\frac{6(\max(1,\beta^2)+L^2_{\nabla f})}{\beta}\Big[W(x^0,z^0)\\
{\quad}  +\frac{L_{\nabla \eta}}{2}\|\nabla F(x^0)-z^0\|^2+(L_{\nabla \eta}+2)\hat \sigma^2_s\Big]+2\sigma^2_s\bigg),
\end{multline}
where the distribution of $R$ is still given by \eqref{def_probl}.
\end{theorem}

\begin{proof}
First, note that by the choice of stepsizes in \eqref{def_tau}, condition \eqref{const_cond1_new} is satisfied with $\hat c=1$. Moreover, by \eqref{dk_expect_new} and \eqref{def_beta_k_new}, we have
\begin{align*}
&\sum_{k=0}^{N-1} \tau_k \Eb[\|d^k\|^2 | \Fc_{k-1}]  \le \frac{6}{\beta} \left[W(x^0,z^0)+\frac{L_{\nabla \eta}}{2}\|\nabla F(x^0)-z^0\|^2+(L_{\nabla \eta}+2)\hat \sigma^2_s\right], \nn \\
&\sum_{k=1}^{N} \tau_k \Eb[\|\nabla F(x^k)-z^k\|^2| \Fc_{k-1}] \le L_{\nabla f}^2 \sum_{k=0}^{N-1} \tau_k \Eb[\|d^k\|^2 | \Fc_k]+2\hat \sigma^2_s.
\end{align*}
Combining the above relations with \eqref{def_Vopt}, we obtain \eqref{main_rec3_new2}.
\end{proof}

\vgap

While the rate of convergence of Algorithm~\ref{alg_ASA} in \eqref{main_rec3_new2} is of the same order as  in \eqref{main_rec3_new}, the former is obtained
under a relaxed assumption on the outputs of the stochastic oracle, as stated in Assumption~\ref{var_stoch_assump}. However, in this case, the regularization coefficient $\beta_k$ depends on the problem parameters (like in other algorithms for smooth stochastic optimization).

\vgap
 We also have the following asymptotic convergence result.

\begin{theorem}
\label{t:convergence-sl}
Assume that the sequence of stepsizes satisfy \eqref{general_cond}.
Then a constant $\bar{a}>0$ exists  such that, for all $a\in (0,\bar{a})$,  with probability 1, every accumulation point $(x^*,z^*)$ of the sequence $\{x^k,z^k)$ generated by Algorithm \ref{alg_ASA} satisfies the conditions:
\begin{gather*}
z^* = \nabla f(x^*),\\
-z^* \in {\cal N}_X(x^*).
\end{gather*}
\end{theorem}
\begin{proof}
 The  analysis follows from \cite{ruszczynski87}.
\end{proof}

\vgap

\section{Concluding Remarks}

We have presented a single time-scale stochastic approximation method for smooth nested optimization problems. We showed that the sample complexity bound of this method for finding an approximate stationary point of the problem is in the same order as that of the best-known bound for the stochastic gradient method for single-level stochastic optimization problems. Furthermore, our convergence analysis is the same for both unconstrained and constrained cases and does not require  batches of samples per iteration. We also presented a simplified parameter-free variant of the NASA method for single level problems, which enjoys the same complexity bound, regardless of the existence of constraints.


\end{document}